\documentclass[reqno,12pt]{amsart}
\usepackage[margin=1in]{geometry}
\usepackage[T1]{fontenc}
\usepackage[utf8]{inputenc}
\usepackage{amsmath,amssymb,amsthm}
\usepackage{microtype}
\usepackage[hyphens]{url}
\usepackage{hyperref}
\usepackage{xcolor}
\usepackage{graphicx}
\usepackage{caption}
\usepackage{tikz}
\usetikzlibrary{decorations.pathreplacing}

\hypersetup{
  colorlinks=true,
  linkcolor=blue!50!black,
  citecolor=blue!50!black,
  urlcolor=blue!50!black,
}

\newtheorem{theorem}{Theorem}[section]
\newtheorem{corollary}[theorem]{Corollary}
\newtheorem{proposition}[theorem]{Proposition}
\newtheorem{lemma}[theorem]{Lemma}

\theoremstyle{definition}
\newtheorem{definition}[theorem]{Definition}
\newtheorem{example}[theorem]{Example}

\theoremstyle{remark}
\newtheorem{remark}[theorem]{Remark}

\numberwithin{equation}{section}

\newcommand{\RR}{\mathbb R}
\newcommand{\one}{\mathbf 1}
\newcommand{\diag}{\operatorname{diag}}
\newcommand{\tr}{\operatorname{tr}}
\newcommand{\adjug}{\operatorname{adj}}
\newcommand{\epsov}{\varepsilon}

\makeatletter
\def\thmhead@plain#1#2#3{%
  \thmname{#1}\thmnumber{\@ifnotempty{#1}{ }\@upn{#2}}%
  \thmnote{ \the\thm@notefont(#3)}}%
\let\thmhead\thmhead@plain
\makeatother

\title[The Ferrers bound for spanning trees in bipartite graphs]
{The Ferrers bound for \\spanning trees in bipartite graphs}
\author{Boon Suan Ho}
\address{Department of Mathematics, National University of Singapore}
\email{\href{mailto:hbs@u.nus.edu}{hbs@u.nus.edu}}
\date{}

\begin{document}
\begin{abstract}
We prove Ehrenborg's conjecture that every connected bipartite graph~$G$ with
parts of size~$m$ and~$n$ has at most $\frac{1}{mn}\prod_{v\in V(G)}\deg(v)$
spanning trees, and that equality holds if and only if $G$ is a Ferrers graph.
\end{abstract}

\maketitle

\section{Introduction}

Let~$G$ be a bipartite graph with bipartition $V(G)=X\sqcup Y$, where
$|X|=m\ge 1$ and $|Y|=n\ge 1$, and let~$\tau(G)$ denote the number of spanning
trees of~$G$. The \emph{Ferrers invariant} of~$G$ is the number
\begin{equation}\label{eq:F}
F(G)=\frac{1}{mn}\prod_{v\in V(G)}\deg(v).
\end{equation}
We write $N(v)$ for the set of neighbors of a vertex~$v$. Following Ehrenborg
and van~Willigenburg~\cite{EhrenborgVW}, we call $G$ a \emph{Ferrers graph} if
its vertices can be labeled $X=\{x_1,\dots,x_m\}$ and $Y=\{y_1,\dots,y_n\}$ in
such a way that, for some integers $t_1,\dots,t_n$ with
$m=t_1\ge t_2\ge\dots\ge t_n\ge 1$, the neighborhoods of the vertices of~$Y$
are given by
\begin{equation}\label{eq:ferrers-def}
N(y_j)=\{x_1,\dots,x_{t_j}\}\qquad(1\le j\le n).
\end{equation}
The edges~$x_iy_j$ of a Ferrers graph are thus the boxes of a Ferrers diagram
whose columns have heights $t_1,\dots,t_n$; see Figure~\ref{fig:ferrers-graph}.
Equivalently, a bipartite graph is a Ferrers graph if and only if it is
connected and the neighborhoods~$N(y)$ of the vertices $y\in Y$ are totally
ordered by inclusion. (To obtain a labeling satisfying
condition~\eqref{eq:ferrers-def} from such a chain of neighborhoods, order the vertices
of~$X$, and those of~$Y$, by weakly decreasing degree.) The definition is
symmetric in~$X$ and~$Y$, because the neighborhoods~$N(x)$ of the vertices
$x\in X$ of a Ferrers graph also form a chain.

Ehrenborg and van~Willigenburg proved that every Ferrers graph~$G$ satisfies
$\tau(G)=F(G)$~\cite{EhrenborgVW}. Ehrenborg conjectured in 2006 that $F(G)$
is an upper bound for the number of spanning trees of every bipartite graph~$G$ (see~\cite{Slone}). We prove this conjecture, and we show that the Ferrers
graphs are the only connected bipartite graphs for which the bound is attained.

\begin{theorem}\label{thm:main}
Let~$G$ be a connected bipartite graph whose parts are nonempty and have
sizes~$m$ and~$n$.
Then the number of spanning trees of~$G$ satisfies
\[
\tau(G)\le \frac{1}{mn}\prod_{v\in V(G)}\deg(v).
\]
Moreover, equality holds if and only if $G$ is a Ferrers graph.
\end{theorem}

A disconnected graph has no spanning trees, so the inequality
$\tau(G)\le F(G)$ holds for all bipartite graphs.

\begin{figure}[htbp]
    \centering
    \includegraphics[width=0.7\textwidth]{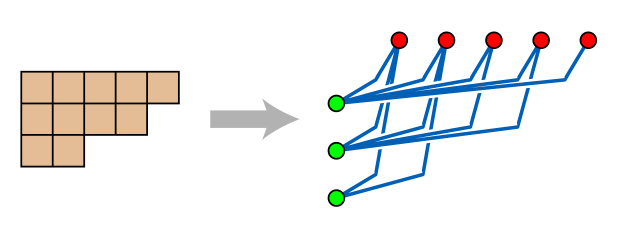}
    \caption{The edges of a Ferrers graph are in direct correspondence with
    the boxes in a Ferrers diagram. (Reproduced from David Eppstein,
    ``Four open from IPAM''~\cite{EppsteinIPAM}, 2009, CC BY 4.0 license.)}
    \label{fig:ferrers-graph}
\end{figure}

\subsection*{Earlier work}
The formula $\tau(G)=F(G)$ for Ferrers graphs was proved in~\cite{EhrenborgVW} in a weighted form, using ideas from the theory of
electrical networks. It has since been proved again in several ways: Burns
gave bijective proofs~\cite{Burns}, Bapat gave a proof via resistance
distance~\cite{Bapat}, Volkova computed the relevant determinant
inductively~\cite{Volkova}, and Klee and Stamps~\cite{KleeStamps} and Go,
Khwa, Luo, and Stamps~\cite{GoKhwaLuoStamps} derived the formula from rank-one
perturbations of Laplacian matrices.

Ehrenborg's conjecture had been established in several special cases. Garrett
and Klee~\cite{GarrettKlee} reformulated it as the nonnegativity of a certain
homogeneous polynomial, and used this reformulation to verify the conjecture
when one part has size at most~$5$; they also showed that adding a pendant
vertex preserves the bound $\tau(G)\le F(G)$, which settles the case of
trees. Koo~\cite{Koo} proved the conjecture for even cycles and for
sufficiently dense bipartite graphs with a cut vertex of degree~$2$, and
showed that the bound is preserved when two graphs are joined by a new edge.
Slone~\cite{Slone} showed that the bound is preserved when two graphs are
glued along a vertex, so that a minimal counterexample would have to be
$2$-connected; the same
paper recorded a computer verification of the conjecture for graphs with at
most $13$~vertices, a range later extended to $17$~vertices~\cite{Slone2}, and proposed a stronger majorization-based statement
that was shown there to be false. Volkova~\cite{Volkova} proved the conjecture
for bipartite graphs in which all vertices of one part have the same degree.

In a different direction, Bozkurt~\cite{Bozkurt} proved that every connected
bipartite graph~$G$ satisfies
\[
\tau(G)\le \frac{1}{|E(G)|}\prod_{v\in V(G)} \deg(v),
\]
with equality exactly for complete bipartite graphs; this bound is weaker
than the Ferrers bound by the factor~$mn/|E(G)|$. Spectral and majorization
reformulations of the conjecture were recorded by Bapat~\cite{Bapat} and
Slone~\cite{Slone}. More recently, Cherkashin and
Prozorov~\cite{CherkashinProzorov} showed that the conjecture is equivalent to
a polynomial inequality in vertex weights; we state this consequence of
Theorem~\ref{thm:main} as Corollary~\ref{cor:polynomial}.

\subsection*{Outline of the proof}
Our proof is linear-algebraic, in the spirit of Klee and
Stamps~\cite{KleeStampsWeighted,KleeStamps}. Write $X=\{x_1,\dots,x_m\}$
and $a_i=\deg(x_i)$. In
Section~\ref{sec:reduction}, a rank-one update of the Laplacian matrix of~$G$
followed by a Schur complement expresses $\tau(G)$ through the determinant of
a positive semidefinite $m\times m$ matrix~$M$, and reduces
Theorem~\ref{thm:main} to the inequality $\det M\le a_1a_2\dots a_m$. The key
observation, made in Section~\ref{sec:projections}, is that $M$ is a sum of
orthogonal projections,
\[
M=\sum_{y\in Y}Q_{N(y)},
\]
where $Q_T$ is an explicit orthogonal projection of rank~$|T|$ for each
nonempty set $T\subseteq X$. These projections satisfy the trace identity
\[
\tr(Q_IQ_T)=|I\cap T|+\frac{|I\setminus T|\,|T\setminus I|}{|I|\,|T|}.
\]
The last term, which we call the \emph{overlap defect} of~$I$ and~$T$, is an
error term that is nonnegative and vanishes precisely when $I\subseteq T$ or
$T\subseteq I$.

In Section~\ref{sec:proof} we label the vertices of~$X$ so that
$a_1\ge a_2\ge\dots\ge a_m$. The projections $Q_{\{x_1,\dots,x_k\}}$ are then
simultaneously diagonal in an explicit orthonormal basis $h_1,\dots,h_m$ of~$\RR^m$, and the trace identity computes the diagonal of~$M$ in this basis:
\[
h_k^\top Mh_k=a_k+\delta_k-\delta_{k-1},
\]
where $\delta_k\ge 0$ is the sum of the overlap defects between
$\{x_1,\dots,x_k\}$ and the neighborhoods~$N(y)$. By Hadamard's inequality, $\det M$
is at most the product of these diagonal entries, and the inequality of
arithmetic and geometric means shows that this product is at most
$a_1a_2\dots a_m$. Equality forces every~$\delta_k$ to vanish, so that every
neighborhood~$N(y)$ is comparable with every set $\{x_1,\dots,x_k\}$, and this
happens only for Ferrers graphs. In Section~\ref{sec:further} we show that the
same computation implies that the eigenvalues of~$M$ majorize the degree
sequence $(a_1,\dots,a_m)$, and we state the polynomial form of
Theorem~\ref{thm:main}.

\section{Reduction to a determinant inequality}\label{sec:reduction}

Throughout the rest of the paper, $G$ is a connected bipartite graph with
bipartition $V(G)=X\sqcup Y$, where $X=\{x_1,\dots,x_m\}$ and
$Y=\{y_1,\dots,y_n\}$, and we write
\begin{equation}\label{eq:degrees}
a_i=\deg(x_i)\quad(1\le i\le m)
\qquad\text{and}\qquad
c_j=\deg(y_j)\quad(1\le j\le n).
\end{equation}
These degrees are positive, because the graph~$G$ is connected and has at
least two
vertices. In this notation, the Ferrers invariant of equation~\eqref{eq:F} is
$F(G)=\frac{1}{mn}\,a_1\dots a_m\,c_1\dots c_n$.

All vectors are column vectors. We write $\one$ for the all-ones vector,
whose length will be clear from the context, and $J=\one\one^\top$ for the
all-ones $m\times m$ matrix. The \emph{Laplacian matrix} $L(H)$ of a graph~$H$
has rows and columns indexed by the vertices of~$H$; its diagonal entry at a
vertex~$v$ is~$\deg(v)$, and its entry at a pair of distinct vertices $v,w$ is~$-1$ if $vw$ is an edge, and~$0$ otherwise. If we list the vertices of our
bipartite graph~$G$ in the order $x_1,\dots,x_m,y_1,\dots,y_n$, then its
Laplacian matrix has the block form
\begin{equation}\label{eq:block-form}
L(G)=
\begin{pmatrix}
A & -B\\
-B^\top & C
\end{pmatrix},
\qquad
A=\diag(a_1,\dots,a_m),\quad C=\diag(c_1,\dots,c_n),
\end{equation}
where $B$ is the $m\times n$ \emph{biadjacency matrix} of~$G$, whose
entry~$B_{ij}$ is~$1$ if $x_iy_j$ is an edge and~$0$ otherwise.

We use three standard tools. The first is the matrix-tree theorem. Recall
that the \emph{adjugate} $\adjug(S)$ of a square matrix~$S$ is the transpose of
its matrix of cofactors.

\begin{theorem}[Kirchhoff's matrix-tree theorem~\cite{Kirchhoff};
see {\cite[Section~13.2]{GodsilRoyle}}]\label{thm:kirchhoff}
Let~$H$ be a graph. Then every cofactor of the Laplacian matrix~$L(H)$ is
equal to~$\tau(H)$; that is, $\adjug(L(H))=\tau(H)\,\one\one^\top$.
\end{theorem}

The second tool is the following consequence, due to Klee and Stamps.

\begin{lemma}[{\cite[Lemma~1]{KleeStamps}}]\label{lem:rank-one}
Let~$H$ be a graph with vertex set~$V$, and let $u,v\in\RR^V$ be vectors
indexed by the vertices of~$H$. Then the
following identity holds:
\[
\det\bigl(L(H)+uv^\top\bigr)=(\one^\top u)(\one^\top v)\,\tau(H).
\]
\end{lemma}

\begin{proof}
The matrix determinant lemma~\cite[Section~0.8.5]{HornJohnson} states that
$\det(S+uv^\top)=\det S+v^\top\adjug(S)\,u$ for every square matrix~$S$ and all vectors $u,v$ of the same size.
We apply it to $S=L(H)$. The rows of~$L(H)$ sum to zero, so $\det L(H)=0$, and
Theorem~\ref{thm:kirchhoff} gives
$v^\top\adjug(L(H))\,u=\tau(H)\,(v^\top\one)(\one^\top u)$.
\end{proof}

The third tool is the Schur complement. Let
\[
S=\begin{pmatrix}S_{11}&S_{12}\\S_{21}&S_{22}\end{pmatrix}
\]
be a square matrix written in block form, where the blocks~$S_{11}$ and~$S_{22}$ are square (so that $S_{12}$ has as many rows as~$S_{11}$ and as many
columns as~$S_{22}$, and similarly for~$S_{21}$), and suppose that $S_{22}$ is
invertible. The matrix $S_{11}-S_{12}S_{22}^{-1}S_{21}$ is called the
\emph{Schur complement} of~$S_{22}$ in~$S$, and it satisfies the determinant
formula
\begin{equation}\label{eq:schur}
\det S=\det(S_{22})\,\det\bigl(S_{11}-S_{12}S_{22}^{-1}S_{21}\bigr);
\end{equation}
see~\cite[Section~0.8.5]{HornJohnson}. This formula follows by taking
determinants in the factorization
\[
\begin{pmatrix}S_{11}&S_{12}\\S_{21}&S_{22}\end{pmatrix}
=
\begin{pmatrix}\mathrm{Id}&S_{12}S_{22}^{-1}\\0&\mathrm{Id}\end{pmatrix}
\begin{pmatrix}S_{11}-S_{12}S_{22}^{-1}S_{21}&0\\S_{21}&S_{22}\end{pmatrix},
\]
in which $\mathrm{Id}$ denotes identity matrices of the appropriate sizes.

We now return to the bipartite graph~$G$. The matrix~$C$ in
the block form~\eqref{eq:block-form} is invertible, because the degrees~$c_j$ are positive.
Let~$L_X$ be the Schur complement of~$C$ in~$L(G)$, and let~$M$ be a
rank-one update of~$L_X$; explicitly, we set
\begin{equation}\label{eq:def-M}
L_X=A-BC^{-1}B^\top,
\qquad
M=L_X+\frac{n}{m}J.
\end{equation}
Both $L_X$ and~$M$ are symmetric $m\times m$ matrices, whose rows and columns
are indexed by the vertices of~$X$.

\begin{proposition}\label{prop:reduction}
Let~$a_i$ and~$c_j$ be the vertex degrees defined in equation~\eqref{eq:degrees},
and let~$M$ be the matrix defined in equation~\eqref{eq:def-M}. Then the number of spanning
trees of~$G$ is given by
\begin{equation}\label{eq:tau-detM}
\tau(G)=\frac{c_1c_2\dots c_n}{mn}\,\det M.
\end{equation}
Consequently, the inequality $\tau(G)\le F(G)$ of Theorem~\ref{thm:main} is
equivalent to the inequality
\begin{equation}\label{eq:det-bound}
\det M\le a_1a_2\dots a_m,
\end{equation}
and equality holds in one of them if and only if it holds in the other.
\end{proposition}

The following proof is due to Steven Klee (personal communication).
Proposition~\ref{prop:reduction} is a symmetric variant of~\cite[Proposition~4]{KleeStamps}, in which a different rank-one update leads
to a non-symmetric matrix in place of~$M$; the symmetry of~$M$ is essential
in what follows.

\begin{proof}
Let $\chi\in\RR^{m+n}$ be the indicator vector of~$X$, whose first $m$
entries are~$1$ and whose last $n$ entries are~$0$. We apply
Lemma~\ref{lem:rank-one} to the graph~$G$ with $u=n\chi$ and
$v=\frac1m\chi$, which satisfy $\one^\top u=mn$ and $\one^\top v=1$. We
obtain
\begin{align*}
mn\,\tau(G)
&=\det\Bigl(L(G)+\frac nm\,\chi\chi^\top\Bigr)
 &&\text{by Lemma~\ref{lem:rank-one}}\\
&=\det\begin{pmatrix}
A+\frac nmJ & -B\\[3pt]
-B^\top & C
\end{pmatrix}
 &&\text{by the block form~\eqref{eq:block-form}}\\
&=\det(C)\,\det\Bigl(A+\frac nmJ-BC^{-1}B^\top\Bigr)
 &&\text{by the determinant formula~\eqref{eq:schur}}\\
&=c_1c_2\dots c_n\,\det M
 &&\text{by the definition~\eqref{eq:def-M}.}
\end{align*}
This proves equation~\eqref{eq:tau-detM}. The remaining assertions follow from
equation~\eqref{eq:tau-detM}, because
$F(G)=\frac{1}{mn}\,a_1\dots a_m\,c_1\dots c_n$ and the degrees~$c_j$ are
positive.
\end{proof}

\section{Projections attached to neighborhoods}\label{sec:projections}

Recall that the graph~$G$ is connected and bipartite, with bipartition
$X\sqcup Y$,
and that the matrices~$L_X$ and~$M$ were defined in
equation~\eqref{eq:def-M}. All matrices in this section are $m\times m$, with
rows and columns indexed by $x_1,\dots,x_m$, and all vectors lie in~$\RR^m$.

For every nonempty set $T\subseteq X$, let $\one_T\in\RR^m$ be its indicator
column vector, and define the three $m\times m$ matrices
\begin{equation}\label{eq:def-P}
D_T=\diag(\one_T),\qquad
J_T=\one_T\one_T^{\top},\qquad
P_T=D_T-\frac{1}{|T|}J_T.
\end{equation}
For a vector $u\in\RR^m$, write
$\bar u_T=\frac{1}{|T|}\sum_{x_i\in T}u_i$ for the average of its coordinates
over~$T$. The vector~$P_Tu$ is obtained from~$u$ by subtracting this average
on~$T$ and replacing the coordinates outside~$T$ by zero:
\begin{equation}\label{eq:PTu}
(P_Tu)_i=
\begin{cases}
u_i-\bar u_T& \text{if } x_i\in T,\\
0& \text{if } x_i\notin T.
\end{cases}
\end{equation}
We also define the subspace
\begin{equation}\label{eq:def-H}
H_T=\Bigl\{u\in\RR^m:\ u_i=0\text{ whenever }x_i\notin T,\text{ and }
\sum_{x_i\in T}u_i=0\Bigr\},
\end{equation}
which has dimension~$|T|-1$, and the matrix
\begin{equation}\label{eq:def-Q}
Q_T=P_T+\frac{1}{m}J.
\end{equation}

\begin{lemma}\label{lem:projections}
Let $T\subseteq X$ be a nonempty set.
\begin{enumerate}
\item[\textup{(a)}] The matrix~$P_T$ is the orthogonal projection of~$\RR^m$
onto the subspace~$H_T$.
\item[\textup{(b)}] The matrix~$Q_T$ is the orthogonal projection of~$\RR^m$
onto the subspace $H_T\oplus\langle\one\rangle$, where $\langle\one\rangle$ is
the line spanned by~$\one$. In particular, the rank of~$Q_T$ is the
cardinality $|T|$.
\end{enumerate}
\end{lemma}

\begin{proof}
(a) Equation~\eqref{eq:PTu} shows that $P_Tu\in H_T$ for every $u\in\RR^m$,
and that $P_Tw=w$ for every $w\in H_T$, because $\bar w_T=0$. Hence the matrix~$P_T$ is idempotent with image $H_T$. Since this matrix is also symmetric, it
is the
orthogonal projection onto~$H_T$.

(b) The matrix~$\frac1mJ$ is the orthogonal projection onto the line~$\langle\one\rangle$. This line is orthogonal to~$H_T$, because the
coordinates of every vector in~$H_T$ sum to zero. Therefore the sum
$Q_T=P_T+\frac1mJ$ is the orthogonal projection onto
$H_T\oplus\langle\one\rangle$, a subspace of dimension $(|T|-1)+1=|T|$.
\end{proof}

For each vertex $y_j\in Y$ we write
\[
T_j=N(y_j)\subseteq X,
\]
so that $|T_j|=c_j$.

\begin{lemma}\label{lem:sum-of-projections}
The matrices~$L_X$ and~$M$ are sums of orthogonal projections:
\begin{equation}\label{eq:sum-of-projections}
L_X=\sum_{j=1}^n P_{T_j}
\qquad\text{and}\qquad
M=\sum_{j=1}^n Q_{T_j}.
\end{equation}
In particular, the matrix~$M$ is positive semidefinite.
\end{lemma}

\begin{proof}
The $j$th column of~$B$ is~$\one_{T_j}$, and
$C^{-1}=\diag(|T_1|^{-1},\dots,|T_n|^{-1})$, so we have
$BC^{-1}B^\top=\sum_{j=1}^n\frac{1}{|T_j|}J_{T_j}$. Moreover, we have
$A=\sum_{j=1}^nD_{T_j}$, because the vertex~$x_i$ belongs to exactly~$a_i$ of
the sets $T_1,\dots,T_n$. The first identity now follows from the definitions
in equations \eqref{eq:def-M} and~\eqref{eq:def-P}. Adding
$\frac nmJ=\sum_{j=1}^n\frac1mJ$ to both sides yields the second identity, by
the definitions in equations \eqref{eq:def-M} and~\eqref{eq:def-Q}. Finally, orthogonal
projections are positive semidefinite, and so is a sum of them.
\end{proof}

\begin{remark}\label{rem:energy}
The first identity in equation~\eqref{eq:sum-of-projections} has the following
interpretation. Assign weights $u\in\RR^m$ to the vertices of~$X$ and
$v\in\RR^n$ to the vertices of~$Y$. Expanding the squares and comparing with
the block form~\eqref{eq:block-form} shows that the quadratic form of the
Laplacian matrix is
\[
\begin{pmatrix}u\\v\end{pmatrix}^{\!\top}
L(G)\begin{pmatrix}u\\v\end{pmatrix}
=\sum_{i=1}^ma_iu_i^2-2u^\top Bv+\sum_{j=1}^nc_jv_j^2
=\sum_{j=1}^n\,\sum_{x_i\in T_j}(u_i-v_j)^2.
\]
For fixed $u$, the inner sum is minimized when the weight~$v_j$ of the vertex~$y_j$ equals the average weight~$\bar u_{T_j}$ of the vertices in~$T_j$, and
by equation~\eqref{eq:PTu} and Lemma~\ref{lem:projections}(a) its minimum
value is $\sum_{x_i\in T_j}(u_i-\bar u_{T_j})^2=u^\top P_{T_j}u$. Summing over~$j$ and using Lemma~\ref{lem:sum-of-projections}, we see that $u^\top L_Xu$ is
the minimum of the quadratic form of~$L(G)$ over all choices of~$v$.
\end{remark}

\begin{definition}\label{def:overlap}
The \emph{overlap defect} of two nonempty sets $I,T\subseteq X$ is the number
\begin{equation}\label{eq:overlap-defect}
\epsov(I,T)=\frac{|I\setminus T|}{|I|}\cdot\frac{|T\setminus I|}{|T|}.
\end{equation}
\end{definition}

The overlap defect is the proportion of~$I$ lying outside~$T$, multiplied by
the proportion of~$T$ lying outside~$I$. It is symmetric in~$I$ and~$T$, it
satisfies $0\le\epsov(I,T)\le 1$, and it vanishes exactly for nested sets:
\begin{equation}\label{eq:defect-zero}
\epsov(I,T)=0\quad\text{if and only if one of the inclusions }
I\subseteq T\text{ or }T\subseteq I\text{ holds.}
\end{equation}
The overlap defect thus measures the failure of~$I$ and~$T$ to be comparable
by inclusion. It enters the proof through the following identity.

\begin{lemma}\label{lem:overlap}
For all nonempty sets $I,T\subseteq X$, the following identity holds:
\begin{equation}\label{eq:overlap}
\tr(Q_IQ_T)=|I\cap T|+\epsov(I,T).
\end{equation}
\end{lemma}

\begin{proof}
Write $p=|I|$, $q=|T|$, and $r=|I\cap T|$. Equation~\eqref{eq:PTu} gives
$P_I\one=P_T\one=0$, hence we have $\tr(P_IJ)=\one^\top P_I\one=0$ and
likewise $\tr(JP_T)=0$. Since $\tr(J^2)=m^2$, the definition~\eqref{eq:def-Q} yields
\begin{equation}\label{eq:trace-QQ}
\tr(Q_IQ_T)
=\tr(P_IP_T)+\frac1m\tr(P_IJ)+\frac1m\tr(JP_T)+\frac{1}{m^2}\tr(J^2)
=\tr(P_IP_T)+1.
\end{equation}
Next we expand the product~$P_IP_T$ using the definition~\eqref{eq:def-P}.
The four traces that occur are
\[
\tr(D_ID_T)=r,\quad
\tr(D_IJ_T)=\one_T^\top D_I\one_T=r,\quad
\tr(J_ID_T)=r,\quad
\tr(J_IJ_T)=(\one_I^\top\one_T)^2=r^2,
\]
and therefore we obtain
\begin{equation}\label{eq:trace-PP}
\tr(P_IP_T)
=r-\frac rq-\frac rp+\frac{r^2}{pq}
=r-1+\Bigl(1-\frac rp\Bigr)\Bigl(1-\frac rq\Bigr)
=r-1+\frac{(p-r)(q-r)}{pq}.
\end{equation}
Since $p-r=|I\setminus T|$ and $q-r=|T\setminus I|$, the last fraction is
$\epsov(I,T)$, and equation~\eqref{eq:overlap} follows by combining
equations \eqref{eq:trace-QQ} and~\eqref{eq:trace-PP}.
\end{proof}

\begin{corollary}\label{cor:trace-QM}
For every nonempty set $I\subseteq X$, the following identity holds:
\begin{equation}\label{eq:trace-QM}
\tr(Q_IM)=\sum_{x_i\in I}a_i+\sum_{j=1}^n\epsov(I,T_j).
\end{equation}
\end{corollary}

\begin{proof}
By Lemma~\ref{lem:sum-of-projections} and Lemma~\ref{lem:overlap} we have
\[
\tr(Q_IM)=\sum_{j=1}^n\tr(Q_IQ_{T_j})
=\sum_{j=1}^n|I\cap T_j|+\sum_{j=1}^n\epsov(I,T_j).
\]
The sum $\sum_{j=1}^n|I\cap T_j|$ counts the edges~$x_iy_j$ of~$G$ with
$x_i\in I$, and hence it equals $\sum_{x_i\in I}a_i$.
\end{proof}

\begin{figure}[htbp]
\centering
\begin{tikzpicture}[
  vx/.style={circle,draw,fill=black!75,inner sep=1.8pt},
  vy/.style={circle,draw,fill=white,inner sep=1.8pt}]
\foreach \i in {1,2,3} {
  \node[vx,label=left:$x_\i$] (x\i) at (0,-0.9*\i) {};
  \node[vy,label=right:$y_\i$] (y\i) at (2.4,-0.9*\i) {};
}
\draw (x1)--(y1) (x1)--(y2) (x1)--(y3) (x2)--(y1) (x2)--(y2) (x3)--(y3);
\begin{scope}[shift={(6.2,-0.45)}]
  \foreach \r/\c in {1/1,1/2,1/3,2/1,2/2}
    \filldraw[fill=black!45,draw=black]
      (0.9*\c-0.9,-0.9*\r) rectangle ++(0.9,0.9);
  \filldraw[fill=black!15,draw=black] (1.8,-2.7) rectangle ++(0.9,0.9);
  \draw[densely dashed] (1.8,-1.8) rectangle ++(0.9,0.9);
  \foreach \r in {1,2,3} \node at (-0.55,-0.9*\r+0.45) {$x_\r$};
  \foreach \c in {1,2,3} \node at (0.9*\c-0.45,0.35) {$y_\c$};
  \draw[decorate,decoration={brace,amplitude=4pt}]
    (3.0,0) -- (3.0,-1.8) node[midway,right=5pt] {$I=\{x_1,x_2\}$};
  \draw[decorate,decoration={brace,amplitude=4pt,mirror}]
    (1.8,-2.85) -- (2.7,-2.85) node[midway,below=5pt] {$T_3$};
\end{scope}
\end{tikzpicture}
\caption{The graph~$G_0$ of Example~\ref{ex:running} and its diagram, which
has a box in row~$x_i$ and column~$y_j$ whenever $x_iy_j$ is an edge. The
diagram is not a Ferrers diagram: The column $T_3=\{x_1,x_3\}$ is not
comparable with the set $I=\{x_1,x_2\}$ of the first two rows (dark boxes),
because it misses the dashed box in row~$x_2$ and contains the light box in
row~$x_3$.}
\label{fig:example}
\end{figure}

\begin{example}\label{ex:running}
Let~$G_0$ be the graph with $m=n=3$ drawn in Figure~\ref{fig:example}, whose
six edges are~$x_1y_1$, $x_1y_2$, $x_1y_3$, $x_2y_1$, $x_2y_2$, and~$x_3y_3$.
Its degrees are $(a_1,a_2,a_3)=(3,2,1)$ and $(c_1,c_2,c_3)=(2,2,2)$, and the
neighborhoods of the vertices of~$Y$ are
\[
T_1=T_2=\{x_1,x_2\}\qquad\text{and}\qquad T_3=\{x_1,x_3\}.
\]
The sets~$T_1$ and~$T_3$ are not comparable, so the graph~$G_0$ is not a Ferrers
graph.
The definition~\eqref{eq:def-P} gives
\[
P_{T_1}=P_{T_2}=\frac12
\begin{pmatrix}1&-1&0\\-1&1&0\\0&0&0\end{pmatrix}
\qquad\text{and}\qquad
P_{T_3}=\frac12
\begin{pmatrix}1&0&-1\\0&0&0\\-1&0&1\end{pmatrix},
\]
and since $\frac nm=1$, Lemma~\ref{lem:sum-of-projections} yields
\[
M=2P_{T_1}+P_{T_3}+J=\frac12
\begin{pmatrix}5&0&1\\0&4&2\\1&2&3\end{pmatrix},
\qquad\text{so that}\qquad \det M=\frac92.
\]
Proposition~\ref{prop:reduction} therefore gives
$\tau(G_0)=\frac{2\cdot2\cdot2}{3\cdot3}\cdot\frac92=4$, in agreement with a
direct count: The graph~$G_0$ consists of the $4$-cycle~$x_1y_1x_2y_2$ and
the pendant path~$x_1y_3x_3$. The Ferrers invariant is
$F(G_0)=\frac{48}{9}=\frac{16}{3}$, and indeed
$\det M=\frac92<6=a_1a_2a_3$.

For the set $I=\{x_1,x_2\}$, the only nonzero overlap defect in
equation~\eqref{eq:trace-QM} is
$\epsov(I,T_3)=\frac12\cdot\frac12=\frac14$, because $x_2\in I\setminus T_3$
and $x_3\in T_3\setminus I$. Accordingly,
$\tr(Q_IM)=(3+2)+\frac14=\frac{21}{4}$.
\end{example}

\section{Proof of the main theorem}\label{sec:proof}

We now label the vertices of~$X$ in such a way that
\begin{equation}\label{eq:ordering}
a_1\ge a_2\ge\dots\ge a_m,
\end{equation}
and for $0\le k\le m$ we write $[k]=\{x_1,\dots,x_k\}$ for the set of the
first $k$ vertices of~$X$. Let $e_1,\dots,e_m$ be the standard basis of~$\RR^m$, and define the vectors
\begin{equation}\label{eq:basis}
h_1=\frac{1}{\sqrt m}\,\one
\qquad\text{and}\qquad
h_k=\frac{1}{\sqrt{k(k-1)}}\bigl(\one_{[k-1]}-(k-1)e_k\bigr)
\quad(2\le k\le m).
\end{equation}

\begin{lemma}\label{lem:basis}
The vectors $h_1,\dots,h_m$ form an orthonormal basis of~$\RR^m$, and the
projections~$Q_{[k]}$ are given by
\begin{equation}\label{eq:Q-flag}
Q_{[k]}=\sum_{i=1}^kh_ih_i^\top\qquad(1\le k\le m).
\end{equation}
\end{lemma}

\begin{proof}
Each vector~$h_k$ has norm~$1$; for $k\ge 2$ this holds because
$\lVert\one_{[k-1]}-(k-1)e_k\rVert^2=(k-1)+(k-1)^2=k(k-1)$. For $k\ge 2$, the
vector~$h_k$ is supported on~$[k]$, its coordinates sum to zero, and it is
constant on~$[k-1]$. Hence the vector~$h_k$ is orthogonal to~$h_1$, which is constant,
and to each~$h_i$ with $2\le i<k$, which is supported on~$[k-1]$ and has
coordinate sum zero. Thus the vectors
$h_1,\dots,h_m$ form an orthonormal basis.

Now fix~$k$. The vectors $h_2,\dots,h_k$ are $k-1$ orthonormal vectors in the
subspace~$H_{[k]}$ of equation~\eqref{eq:def-H}, which has dimension~$k-1$.
Therefore the vectors $h_1,\dots,h_k$ form an orthonormal basis of
$H_{[k]}\oplus\langle\one\rangle$. The orthogonal projection onto this
subspace is $\sum_{i=1}^kh_ih_i^\top$, and by Lemma~\ref{lem:projections}(b)
it is also $Q_{[k]}$, which proves equation~\eqref{eq:Q-flag}.
\end{proof}

The next lemma computes the diagonal entries of the matrix~$M$ with respect
to the basis $h_1,\dots,h_m$. It is valid for every labeling of~$X$.

\begin{lemma}\label{lem:diagonal}
Let $\delta_0=0$, and for $1\le k\le m$ let~$\delta_k$ be the sum of the
overlap defects between the set~$[k]$ and the neighborhoods $T_j=N(y_j)$:
\begin{equation}\label{eq:delta}
\delta_k=\sum_{j=1}^n\epsov\bigl([k],T_j\bigr).
\end{equation}
Then the numbers~$\delta_k$ are nonnegative, the last one~$\delta_m$ is zero,
and the following identity holds:
\begin{equation}\label{eq:diagonal}
h_k^\top Mh_k=a_k+\delta_k-\delta_{k-1}\qquad(1\le k\le m).
\end{equation}
\end{lemma}

\begin{proof}
The numbers~$\delta_k$ are nonnegative because overlap defects are, and
$\delta_m=0$ by the equivalence~\eqref{eq:defect-zero}, because $T_j\subseteq[m]=X$ for every~$j$. For $1\le k\le m$, equation~\eqref{eq:Q-flag}, the identity
$\tr(hh^\top M)=h^\top Mh$ for vectors $h\in\RR^m$, and
Corollary~\ref{cor:trace-QM} yield
\[
\sum_{i=1}^kh_i^\top Mh_i=\tr\bigl(Q_{[k]}M\bigr)
=\sum_{i=1}^ka_i+\delta_k.
\]
The equality between the first and the last expression also holds for $k=0$,
when both are zero. Subtracting the
identity for~$k-1$ from the identity for~$k$ gives
equation~\eqref{eq:diagonal}.
\end{proof}

\begin{example}\label{ex:running-2}
For $m=3$, the basis of equation~\eqref{eq:basis} is
\[
h_1=\tfrac{1}{\sqrt3}(1,1,1)^\top,\qquad
h_2=\tfrac{1}{\sqrt2}(1,-1,0)^\top,\qquad
h_3=\tfrac{1}{\sqrt6}(1,1,-2)^\top.
\]
For the graph~$G_0$ of Example~\ref{ex:running}, whose labeling satisfies
the ordering~\eqref{eq:ordering}, we have $\delta_1=0$, because $x_1$ lies in
every set~$T_j$; $\delta_2=\frac14$, as computed at the end of that example;
and
$\delta_3=0$. Lemma~\ref{lem:diagonal} therefore predicts the diagonal entries
\[
\bigl(h_1^\top Mh_1,\ h_2^\top Mh_2,\ h_3^\top Mh_3\bigr)
=\bigl(3,\ 2+\tfrac14,\ 1-\tfrac14\bigr).
\]
These values agree with a direct computation from the matrix~$M$ of
Example~\ref{ex:running}; for instance,
$h_2^\top Mh_2=\frac12(e_1-e_2)^\top M(e_1-e_2)
=\frac12\cdot\frac12(5-2\cdot0+4)=\frac94$. The diagonal thus differs from the
degree sequence $(3,2,1)$ by a transfer of $\delta_2=\frac14$ from the third
entry to the second. In general, equation~\eqref{eq:diagonal} says that each
number~$\delta_k$ is added to the $k$th diagonal entry and subtracted from
the $(k+1)$st.
\end{example}

To bound the determinant of~$M$ by these diagonal entries, we use a classical
inequality of Hadamard~\cite{Hadamard}; see~\cite[Theorem~7.8.1]{HornJohnson}.

\begin{theorem}[Hadamard's inequality]\label{thm:hadamard}
Let~$S$ be a positive semidefinite real symmetric $m\times m$ matrix. Then
the determinant of~$S$ is at most the product of its diagonal entries.
\end{theorem}

The inequality is usually stated for positive definite matrices. If the
matrix~$S$ is positive semidefinite and singular, then $\det S=0$, and the
inequality holds because the diagonal entries~$e_i^\top Se_i$ of~$S$ are
nonnegative.

\begin{proof}[Proof of Theorem~\ref{thm:main}]
By Proposition~\ref{prop:reduction}, it suffices to prove the
inequality~\eqref{eq:det-bound}, and to show that equality holds in
inequality~\eqref{eq:det-bound} if and only if $G$ is a Ferrers graph.

\smallskip\noindent\emph{The inequality.}
Let~$U$ be the $m\times m$ matrix whose columns are the vectors
$h_1,\dots,h_m$. By Lemma~\ref{lem:basis} these vectors are orthonormal, so
$U^\top U$ is the identity matrix, and hence $\det(U)^2=1$. The matrix~$U^\top MU$ is symmetric, and it is positive semidefinite because the matrix~$M$ is (Lemma~\ref{lem:sum-of-projections}). Its determinant is
$\det(U)^2\det M=\det M$, and its diagonal entries are the numbers
\begin{equation}\label{eq:def-d}
d_k=h_k^\top Mh_k=a_k+\delta_k-\delta_{k-1}\qquad(1\le k\le m)
\end{equation}
computed in Lemma~\ref{lem:diagonal}. Theorem~\ref{thm:hadamard}, applied to
the matrix~$U^\top MU$, therefore gives
\begin{equation}\label{eq:hadamard}
\det M\le d_1d_2\dots d_m.
\end{equation}
It remains to compare the product of the numbers~$d_k$ with the product of
the degrees~$a_k$. Equation~\eqref{eq:def-d} gives
\begin{equation}\label{eq:sum-ratios}
\sum_{k=1}^m\frac{d_k}{a_k}
=\sum_{k=1}^m\frac{a_k+\delta_k-\delta_{k-1}}{a_k}
=m-\sum_{k=1}^{m-1}\delta_k\Bigl(\frac{1}{a_{k+1}}-\frac{1}{a_k}\Bigr)
\le m.
\end{equation}
Here the second equality uses $\delta_0=\delta_m=0$, and the inequality holds
because $\delta_k\ge 0$ and $a_k\ge a_{k+1}$ by the
ordering~\eqref{eq:ordering}. The numbers~$d_k/a_k$ are nonnegative, because
the matrix~$M$ is positive semidefinite. The inequality of arithmetic and
geometric means, together with inequalities \eqref{eq:hadamard}
and~\eqref{eq:sum-ratios}, therefore yields
\begin{equation}\label{eq:amgm}
\frac{\det M}{a_1a_2\dots a_m}
\le\prod_{k=1}^m\frac{d_k}{a_k}
\le\Bigl(\frac1m\sum_{k=1}^m\frac{d_k}{a_k}\Bigr)^{\!m}
\le 1.
\end{equation}
This proves the inequality~\eqref{eq:det-bound}, and hence Ehrenborg's
conjecture:
\[
\tau(G)=\frac{c_1\dots c_n}{mn}\,\det M
\le\frac{a_1\dots a_m\,c_1\dots c_n}{mn}=F(G).
\]

\smallskip\noindent\emph{Equality implies that the graph~$G$ is a Ferrers graph.}
Suppose that $\det M=a_1a_2\dots a_m$. Then all three inequalities in the
chain~\eqref{eq:amgm} are equalities. Equality in the inequality of
arithmetic and geometric means forces the numbers~$d_k/a_k$ to be equal, and
since these nonnegative numbers have product~$1$, they are all equal to~$1$.
Hence we have $d_k=a_k$ for every~$k$, and
equation~\eqref{eq:def-d} gives $\delta_k=\delta_{k-1}$ for
$1\le k\le m$. Since $\delta_0=0$, every~$\delta_k$ vanishes. By the
definition~\eqref{eq:delta} and the nonnegativity of overlap defects, we
conclude that $\epsov([k],T_j)=0$ for all~$j$ and~$k$, so by
the equivalence~\eqref{eq:defect-zero} every neighborhood~$T_j$ is comparable with every set~$[k]$. In particular, $T_j$ is comparable with the set~$[c_j]$, which has the
same number of elements as~$T_j$. Since two comparable finite sets with the
same number of elements are equal, we obtain
\[
T_j=[c_j]=\{x_1,\dots,x_{c_j}\}\qquad(1\le j\le n).
\]
After relabeling the vertices of~$Y$ so that $c_1\ge\dots\ge c_n$, this is
the condition~\eqref{eq:ferrers-def} with $t_j=c_j$. Here $c_n\ge 1$ because the
degrees are positive, and $c_1=m$ because the vertex~$x_m$ has a neighbor~$y_j$, so that $x_m\in T_j=[c_j]$ and hence $c_j=m$. Thus the graph~$G$
is a Ferrers graph.

\smallskip\noindent\emph{Ferrers graphs attain equality.}
If the graph~$G$ is a Ferrers graph, then $\tau(G)=F(G)$ by the theorem of
Ehrenborg and
van~Willigenburg~\cite{EhrenborgVW}. This also follows from the computation
above. Indeed, label the vertices as in condition~\eqref{eq:ferrers-def}, so that
$T_j=[t_j]$. Then the degree~$a_k$ is the number of indices~$j$ with $t_j\ge k$; in
particular, this labeling satisfies the ordering~\eqref{eq:ordering}.
Lemma~\ref{lem:sum-of-projections} and equation~\eqref{eq:Q-flag} give
\[
M=\sum_{j=1}^nQ_{[t_j]}=\sum_{j=1}^n\,\sum_{k=1}^{t_j}h_kh_k^\top
=\sum_{k=1}^ma_k\,h_kh_k^\top.
\]
The eigenvalues of~$M$ are therefore $a_1,\dots,a_m$, so
$\det M=a_1a_2\dots a_m$, and equality holds in Theorem~\ref{thm:main} by
Proposition~\ref{prop:reduction}.
\end{proof}

\section{Further remarks}\label{sec:further}

\subsection{Majorization}
Combined with a classical inequality of Ky~Fan, Lemma~\ref{lem:diagonal} shows
that the eigenvalues of the matrix~$M$ majorize the degree sequence
$(a_1,\dots,a_m)$. This statement gives more
information about the eigenvalues than their product, and it leads to a second
proof of the inequality~\eqref{eq:det-bound}. Let $\lambda,a\in\RR^m$ be vectors whose coordinates
are in weakly decreasing order. Recall that $\lambda$ \emph{majorizes} $a$ if
the following inequalities hold, with equality for $k=m$:
\[
\sum_{i=1}^k\lambda_i\ge\sum_{i=1}^ka_i\qquad(1\le k\le m).
\]
We use the following classical result of Ky~Fan~\cite{KyFan}; see also~\cite[Exercise~II.1.13]{Bhatia}.

\begin{theorem}[Ky~Fan's maximum principle]\label{thm:ky-fan}
Let~$S$ be a real symmetric $m\times m$ matrix with eigenvalues
$\theta_1\ge\theta_2\ge\dots\ge\theta_m$, and let $w_1,\dots,w_k$ be
orthonormal vectors in~$\RR^m$. Then the following inequality holds:
\[
\sum_{i=1}^kw_i^\top Sw_i\le\sum_{i=1}^k\theta_i.
\]
\end{theorem}

\begin{proposition}\label{prop:majorization}
Label the vertices of~$X$ as in the ordering~\eqref{eq:ordering}, let
$\lambda_1\ge\lambda_2\ge\dots\ge\lambda_m$ be the eigenvalues of the matrix~$M$ of equation~\eqref{eq:def-M}, and let~$\delta_k$ be defined by
equation~\eqref{eq:delta}. Then the
following inequalities hold:
\begin{equation}\label{eq:majorization}
\sum_{i=1}^k\lambda_i\ge\sum_{i=1}^ka_i+\delta_k\qquad(1\le k\le m),
\end{equation}
with equality for $k=m$. In particular, the eigenvalues of~$M$ majorize the
degree sequence $(a_1,\dots,a_m)$.
\end{proposition}

\begin{proof}
Theorem~\ref{thm:ky-fan}, applied to the matrix~$M$ and the orthonormal
vectors $h_1,\dots,h_k$ of equation~\eqref{eq:basis}, gives
$\sum_{i=1}^k\lambda_i\ge\sum_{i=1}^kh_i^\top Mh_i$, and the right-hand side
equals $\sum_{i=1}^ka_i+\delta_k$ by Lemma~\ref{lem:diagonal}. For $k=m$,
both sides of inequality~\eqref{eq:majorization} are equal to~$\tr M$, because
$\delta_m=0$.
\end{proof}

The matrix~$M$ is positive definite: It is positive semidefinite by
Lemma~\ref{lem:sum-of-projections}, and $\det M>0$ by
equation~\eqref{eq:tau-detM}, since the connected graph~$G$ has a spanning
tree. Recall that a function $\phi\colon(0,\infty)\to\RR$ is \emph{concave} if
$\phi(\theta s+(1-\theta)t)\ge\theta\phi(s)+(1-\theta)\phi(t)$ for all
$s,t>0$ and $0\le\theta\le 1$;
the logarithm is an example. By Karamata's inequality~\cite[Section~3.C]{MarshallOlkinArnold},
Proposition~\ref{prop:majorization} implies that
\[
\sum_{i=1}^m\phi(\lambda_i)\le\sum_{i=1}^m\phi(a_i)
\qquad\text{for every concave function }\phi\colon(0,\infty)\to\RR.
\]
The choice $\phi=\log$ gives a second proof of the
inequality~\eqref{eq:det-bound}.

\subsection{A polynomial form}
For a graph~$H$ and a spanning tree~$R$ of~$H$, write $\deg_R(v)$ for the
degree of the vertex~$v$ in~$R$. Cherkashin and Prozorov proved that
Ehrenborg's conjecture is equivalent to an inequality between polynomials in
vertex weights~\cite[Theorem~2]{CherkashinProzorov}.
Theorem~\ref{thm:main} therefore has the following consequence.

\begin{corollary}\label{cor:polynomial}
Let~$H$ be a bipartite graph with bipartition $V(H)=X\sqcup Y$ into nonempty
parts, and define the polynomial
\[
P_H\bigl((z_v)_{v\in V(H)}\bigr)=
\sum_R\,\prod_{v\in V(H)} z_v^{\deg_R(v)-1},
\]
where the sum ranges over all spanning trees~$R$ of~$H$. Then, for all
nonnegative real numbers~$z_v$ $(v\in V(H))$, the following inequality holds:
\[
P_H\bigl((z_v)_{v\in V(H)}\bigr)
\Bigl(\sum_{x\in X} z_x\Bigr)
\Bigl(\sum_{y\in Y} z_y\Bigr)
\le
\prod_{v\in V(H)} \Bigl(\sum_{u\in N(v)} z_u\Bigr).
\]
\end{corollary}

Setting every~$z_v$ equal to~$1$ recovers the inequality
$\tau(H)\le F(H)$.

\section*{Statement on AI use}
GPT-5.4~Pro was used during the development of the original version of this
work to explore possible proof strategies, test intermediate formulations, and
assist with exposition. Claude Fable~5.2 was used to assist with the revision;
in particular, the argument of Section~\ref{sec:proof} based on
Hadamard's inequality and the inequality of arithmetic and geometric means, which replaces the majorization argument of
the original version, was found with its assistance. All mathematical
arguments and claims in the final manuscript were independently verified by
the author, who takes full responsibility for the paper. A Lean~4
formalization of Theorem~\ref{thm:main}, which is based on the majorization
argument of Section~\ref{sec:further}, is available
from \url{https://github.com/boonsuan/FerrersBound}.

\section*{Acknowledgments}
The author thanks Steven Klee for the short proof of
Proposition~\ref{prop:reduction}, Tanay Wakhare for helpful comments, and the
anonymous referees for their careful reading and their many suggestions, which
substantially improved the exposition.

\end{document}